\UseRawInputEncoding
\documentclass[11pt]{amsart}
\usepackage{color,epsfig}      

\newtheorem{thm}{Theorem}

\newtheorem{lem}{Lemma}
\theoremstyle{definition}
\newtheorem{defn}{Definition}
\newtheorem{rem}{Remark}

\newtheorem{cor}{Corollary}

\def\bea{\begin{eqnarray}}
\def\eea{\end{eqnarray}}

\begin{document}

\title[A two-vertex theorem for normal tilings]{A two-vertex theorem for normal tilings}
\author[G. Domokos, \'A.G. Hor\'ath and K. Reg\H os ] {G\'abor Domokos, \'Akos G. Horv\'ath and Krisztina Reg\H os}
\address{G\'abor Domokos, MTA-BME Morphodynamics Research Group and Dept. of Mechanics, Materials and Structures, Budapest University of Technology and Economics,
M\H uegyetem rakpart 1-3., Budapest, Hungary, 1111}
\email{domokos@iit.bme.hu}
\address{\'Akos G. Horv\'ath, Dept. of Geometry, Budapest University of Technology and Economics,
Egry J\'ozsef utca 1., Budapest, Hungary, 1111}
\email{ghorvathakos@gmail.com}
\address{Krisztina Reg\H os,MTA-BME Morphodynamics Research Group, Budapest University of Technology and Economics,
M\H uegyetem rakpart 1-3., Budapest, Hungary, 1111}
\email{regoskriszti@gmail.com}
\thanks{Support of the NKFIH Hungarian Research Fund grant 134199 and of the NKFIH Fund TKP2021 BME-NVA, carried out at the Budapest University of Technology and Economics, is kindly acknowledged.
}
\subjclass[2010]{52C20}

\keywords{normal tiling, sharp vertex}

\maketitle

\begin{abstract}
We regard a smooth, $d=2$-dimensional manifold $\mathcal{M}$ and its normal tiling $M$, the cells of  which
may have non-smooth or smooth vertices (at the latter, two edges meet at 180 degrees.)
We denote the average number (per cell) of non-smooth vertices by $\bar v^{\star}$
and we  prove that  if $M$ is periodic then  $\bar v^{\star} \geq 2$. We show the same result
for the monohedral case by an entirely different argument. Our theory also makes a closely related prediction
for non-periodic tilings. In 3 dimensions we show a monohedral construction with $\bar v^{\star}=0$.
\end{abstract}

\section{Introduction}
Both in nature and among man-made structures we often encounter cellular patterns where a domain is decomposed into mutually disjoint, non-overlapping cells. Such patterns, also called \emph{tessellations} or \emph{tilings,} arise in geology (as fracture patterns \cite{adler_fracture_network_book, Plato, goering_evolv_pattern_2013, goering_columnar_2008,  steacy_automaton_1991}), in engineering  and architecture (as surface patterns or as urban structures \cite{urban1, arch1}), in geography  \cite{map1} or in biology (as cell tissues \cite{cell1, cell2}).

In general, we will assume that the boundaries of cells are piecewise $C^1$-smooth, however, in subsection \ref{ss:monohedral} we will require piecewise $C^2$-smoothness. 
We also assume that the tiling is \emph{normal}\cite{Grunbaum-Shepard}, i.e. that every cell is topologically equivalent to a disk, the intersection of any two cells is a connected set or the empty set,  all cells are uniformly bounded. In addition,  for inifinite tilings we will assume that they are also  \emph{balanced} \cite{Grunbaum-Shepard}, i.e. that the average number of vertices and edges exist independently from the method of the counting, the detailed description of the latter process is given in  Remark \ref{rem1}. Normal tilings have been described in various settings \cite{Horne, Kazanci, Senechal}, however, the emphasis was on their combinatorial properties. Here we aim to describe a curios \emph{metric} aspect of these patterns. 
Cells may have non-smooth or smooth vertices (at the latter, two edges meet at 180 degrees). 
We denote the average number (per cell) of the former by $\bar v^{\star}$ and we 
prove that for periodic patterns $\bar v^{\star} \geq 2$ (see Theorem \ref{main} and Remark \ref{periodic}).
In Theorem \ref{main} we also provide a general formula for the lower bound on $\bar v^{\star}$, including non-periodic patterns on arbitrary, smooth, 2D manifolds. In addition, we prove the $\bar v ^{\star} \geq 2$ statement for monohedral tilings (consisting of congruent cells)
based on a different argument. Finally, we offer a 3-dimensional monohedral construction with $v^{\star}=0$.

\subsection{A simple example and the main result}
We start by giving an informal example, illustrating the basic concepts, based on which we can formulate our main result
and we will provide more formal definitions later.
We regard normal tilings
of smooth, 2D manifolds. 
Figure \ref{fig:0} illustrates three such patterns (illustrated in columns (A),(B),(C), respectively) , all of which  are periodic. In the bottom row we indicate some  real-world examples which could be approximated by these patterns.  We will characterize such patterns by the \emph{combinatorial nodal degree $\bar n$ and combinatorial cell degree $\bar v$} which count the average number of cells meeting at one node and the average number of nodes along the boundary of one cell.   We could also interpret these degrees as counting the number of vertices meeting at a node and counting the number of vertices along the perimeter of a cell. In case of periodic patterns, it is easy to compute these averages; for the particular patterns illustrated in Fig \ref{fig:0} all nodes and cells are identical and we can see that the combinatorial degrees, illustrated and counted  in the upper row, are identical for all three patterns.
\begin{figure}[ht]
\begin{center}
\includegraphics[width=\textwidth]{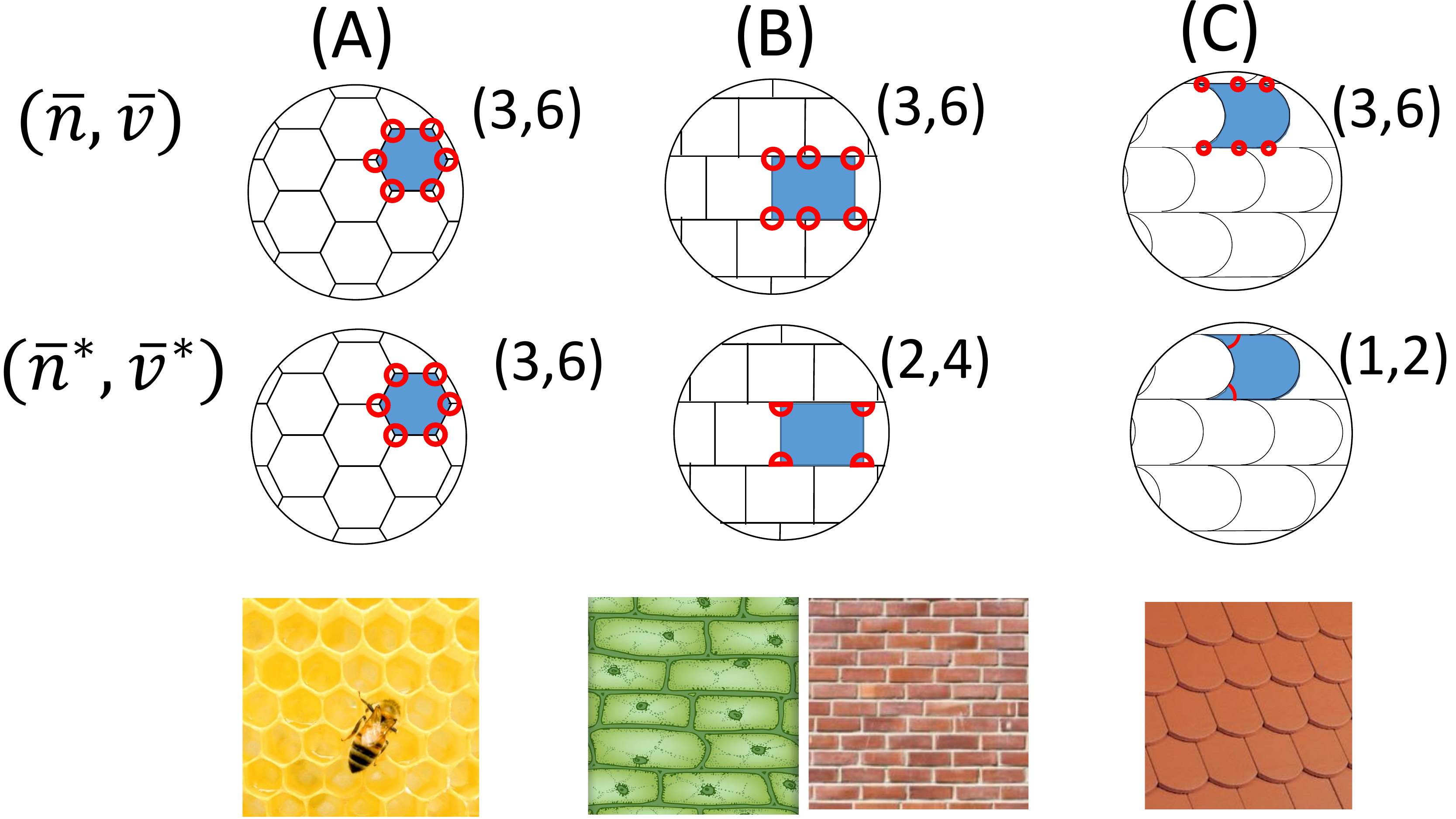}
\caption{Combinatorial and corner degrees in periodic patterns. (A) Hexagonal pattern reminiscent of honeycomb with $(\bar n, \bar v)=(\bar n^{\star}, \bar v^{\star})=(3,6)$. (B) Pattern reminiscent of plant cell tissue or brick wall, with combinatorial degrees $(\bar n, \bar v)=(3,6)$,
corner degrees $(\bar n^{\star}, \bar v^{\star})=(2,4)$. (C) Pattern reminiscent of roof tiles  with combinatorial degrees $(\bar n, \bar v)=(3,6)$,
corner degrees $(\bar n^{\star}, \bar v^{\star})=(1,2)$.}\label{fig:0}
\end{center}
\end{figure}
Despite having identical combinatorial degrees, the three patterns look radically different. One way to capture these
differences is to compute \emph{corner degrees}: the \emph{nodal corner degree} $\bar n ^{\star}$ and the \emph{cell corner degree} $\bar v^{\star}$
count only vertices where the angle of the half-tangents of the meeting arcs is not 180 degrees (and so, obviously, we have $\bar n ^{\star}\leq \bar n$,
$\bar v ^{\star}\leq \bar v$). In case of the hexagonal honeycomb in column (A) of Fig. \ref{fig:0}, the combinatorial and corner degrees coincide since none of the  examined angles is 180 degrees. In case of the brick wall pattern the corner degrees are smaller because we do have vertices with edge-angles 180 degrees. The roof tile pattern has even smaller corner degrees.

While it is clear that corner degrees may be smaller than combinatorial degrees, it is not obvious by how much these quantities may differ. In particular, it is not clear what is the minimum for the cell corner degree which counts the average number \emph{actual, non-smooth} corners of a cell. Our goal is to prove

\begin{thm}\label{main}
Let $\mathcal{M}$ be a smooth, 2D manifold with Euler characteristic
$\chi$ and let $M$ be a { {normal tiling}}
of $\mathcal{M}$ with $E$ edges
and corner degrees $\bar n^{\star}, \bar v^{\star}$. Then
we have
$$
\bar v^{\star} \geq 2-\frac{2\chi\bar n}{E(\bar n-2)+\chi\bar n}.
$$
\end{thm} \noindent As we can see, both in the $E \to \infty$ limit and for $\chi =0$ Theorem \ref{main} yields
\begin{equation}\label{cor1}
\bar v^{\star} \geq 2.
\end{equation}
\begin{rem}\label{periodic}
As periodic patterns can be treated as infinite, equation (\ref{cor1}) holds for such patterns.
\end{rem}

\subsection{Structure of paper}
In Section \ref{2D} we first define the basic concepts and prove Theorem \ref{main}. In Section \ref{3D}  we offer an alternative proof for the 2D monohedral case as well as a construction of a 3D monohedral tiling with no vertices.

\section{Normal tiling of 2D manifolds}\label{2D}

\subsection{Combinatorial properties}

\begin{defn}\label{defav}
Let $M$ be a normal \cite{Grunbaum-Shepard} tiling
of a smooth, $d=2$-dimensional manifold $\mathcal{M}$ and let $M$ have $V$ vertices (0-dimensional cell or nodes), $E$ edges (1-dimensional cells) and $F$ faces (2-dimensional cells or simply cells). We call the number of edges meeting at the $i$th node ($i=1,2, \dots V$) the \emph{combinatorial nodal degree}  and denote it by $n_i$. We call the number of edges  on the boundary of  the $j$th  cell ($j=1,2, \dots F$)
the \emph{combinatorial cell degree}  and denote it by $v_j$. For finite $V,E,F$ we denote the average values of the combinatorial nodal and cell degrees by $\bar n, \bar v$, respectively.
We also allow that $V,E,F$ will be infinite, however, in that case we also assume that $M$ is \emph{balanced} \cite{Grunbaum-Shepard}, guaranteeing that a suitably chosen limit process on a sequence of finite tilings provides the corresponding averages $\bar n, \bar v$ for the nodal and cell degrees (for details see Remark \ref{rem1}). We call the $[\bar n, \bar v]$ plane as the (combinatorial) \emph{symbolic plane}. 
\end{defn}

\begin{rem}\label{rem1}
On compact manifolds and for finite tilings, counting averages is trivial. On infinite domains
averages are again trivial if $M$ is periodic. For general normal tilings
we compute the averages by counting the nodal and cell degrees in a finite ball $B(X, \rho)$ with center $X$ radius $\rho$
to obtain $\bar n(X, \rho), \bar v(X,\rho)$ associated with the finite normal tiling $M(X,\rho)$. Then
we let the radius $\rho$ approach infinity. We say that the averages $\bar n, \bar v$ of the degrees exists if 
the limits $\lim_{\rho \to \infty} \bar n(X,\rho),\lim_{\rho \to \infty} \bar v(X,\rho)$ exist and are independent of $X$.
Henceforth we will assume that the manifold $\mathcal{M}$ is either compact and $M$ is finite, or, $M$ is infinite and
the averages converge in the sense described above and boundary effects decay.
\end{rem}

\begin{rem}\label{rembalance}
Note that on the Euclidean plane the concepts of  a tiling being "normal" and "balanced" are well-investigated (see Section 3.3 in \cite{Grunbaum-Shepard}). Clearly,  the definition of "normality" formally can be applied on arbitrary manifolds, but its precise connection with the property "balanced" is far from obvious, except on the Euclidean plane (see Theorem 3.3.2 of \cite{Grunbaum-Shepard}. In this paper we rely on the strong conditions on averaging, given in Remark \ref{rem1}).
\end{rem}
\begin{rem}\label{rem2} From the fact that the tiling $M$ is normal, it follows that each 2-cell is \emph{contractible}, and the boundary of a 2-cell is a simple closed curve.  We will later discuss $d$-dimensional manifolds and their tilings where we also use that the $d$-dimensional cells are contractible and all $k<d$ dimensional faces are also contractible.
\end{rem}

\begin{rem}\label{rem3} Since all edges are assumed to be $C^1$-smooth,
they have  half-tangents at both   endpoints. We also note that the angle of the two half-tangent{ {s}} can be calculated from the globally defined differential structure of the smooth manifold.
\end{rem}

\begin{defn}\label{defhar}
We call $$\bar h = \frac{\bar v\bar n }{\bar v + \bar n}$$ the \emph{harmonic degree} of $M$.
\end{defn}
For finite $M$, the Euler formula can be written as
\begin{equation}\label{e1}
F - E + V = \chi.
\end{equation}
According to Assumption (1) in Definition \ref{defav}, and the fact that every edge has two vertices and belongs to two cells, the combinatorial nodal and cell
degrees can be expressed as
\begin{equation}\label{ave}
\bar n = \frac{2E}{V}, \quad \bar v = \frac{2E}{F}.
\end{equation}
We can rewrite equation (\ref{e1}) as
\begin{equation}\label{e2}
\frac{F}{2E}+\frac{V}{2E}= \frac{E+\chi }{2E},
\end{equation} which, using equation (\ref{ave}) and Definition \ref{defhar} translates into
\begin{equation}\label{meuler4}
\bar h =\frac{2E}{E+\chi}.
\end{equation}
In the infinite case { {we regard an averaging process analogous to the one given in Remark \ref{rem1}
where the number $E(\rho)$ of edges (and the harmonic degree $\bar h(\rho)$) are defined as the function of the radius $\rho$ of a finite ball. As $\rho \to \infty$, we have
\begin{equation}\label{meuler5}
\bar h :=\lim _{E(\rho) \to \infty} \bar h(\rho) =2,
\end{equation}
regardless of the Euler characteristic of $M=\lim _ {\rho \to \infty} M(\rho)$.
Also, if $\chi=0$ then we have $\bar h =2$,  one example is the torus $T^2$.}}

\begin{rem}\label{rem4}
Observe, that the finiteness of the normal tiling $M$ is equivalent to the compactness of the carrying manifold $\mathcal{M}$.
\end{rem}

\begin{figure}[ht]
\begin{center}
\includegraphics[width=\textwidth]{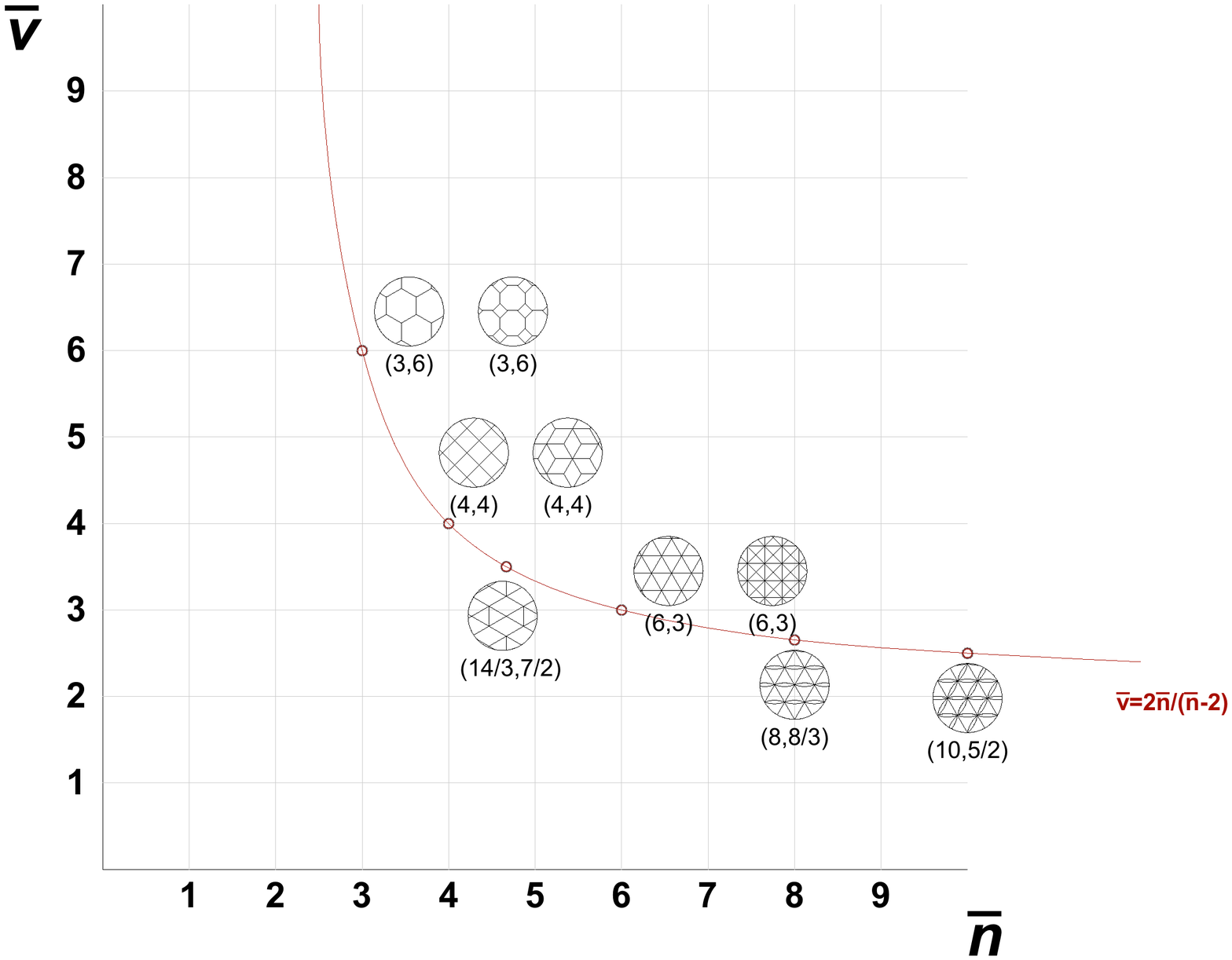}
\caption{Illustration of infinite regular patterns on the symbolic plane $[\bar n, \bar v]$.
All patterns lie on the $\bar h = 2$ hyperbola.
 }\label{fig:CW0}
\end{center}
\end{figure}

\subsection{Geometric properties: sharp corners}

\begin{defn}\label{def:geom_node}
Let $M$ be a normal and balanced tiling
of a smooth, $d=2$-dimensional manifold $\mathcal{M}$ and let $M$ have $V$ vertices (nodes), $E$ edges and $F$ faces (cells) and let $n_i$ denote the combinatorial degree of the $i$th node, $(i=1,2, \dots V)$. We number the edges $e_{i,j}$  ($i=1,2 \dots V, j=1,2, \dots n_i$) meeting at the $i$th node clockwise consecutively using their second subscript, so the \emph{pair} $e_{i,j},e_{i,j+1}$ is not separated by any other edge.
We also regard $e_{i,n_i},e_{i,1}$ as one pair, so we associate $n_i$ pairs with the $i$th node.  We call a pair  \emph{degenerate} if the two edges meet at 180 degrees, otherwise we call the pair \emph{generic}. (The angle between two the edges of an edge-pair is identical to the  alternate interior angle their respective half-tangents at the common vertex, thus we can distinguish between the angle zero (when the two half-tangents have opposite direction on their line) and the angle $\pi$ ( when the two half-tangents have the same direction).) We associate with the $i$th node the quantity $r_i \in \{0,1,2\}$, counting the number of degenerate pairs.
We define the \emph{corner degree} of the $i$th node as  $$n^{\star}_i=n_i-r_i$$.
\end{defn}

\begin{defn}\label{def:geom_cell}
Let $M$ be a normal and balanced tiling
of a smooth, $d=2$-dimensional manifold $\mathcal{M}$  and let $M$ have $V$ vertices (nodes), $E$ edges and $F$ faces (cells) and let $v_i$ denote the combinatorial degree of the $i$th cell, $(i=1,2, \dots F)$.
The edges $e_{i,j}$  ($i=1,2 \dots F, j=1,2, \dots n_i$) surrounding the $i$th cell are numbered clockwise consecutively, so the \emph{pair} $e_{i,j},e_{i,j+1}$ overlaps at a vertex.
We also regard $e_{i,v_i},e_{i,1}$ as one pair, so we associate  $v_i$ pairs with the $i$th cell. We call a pair  \emph{degenerate} if the two edges meet at 180 degrees, otherwise we call the pair \emph{generic}. We associate with the $i$th cell the quantity $q_i \in \{0,1,2, \dots v_i \}$, counting the number of degenerate pairs.
We define the \emph{corner degree} of the $i$th cell as  $$v^{\star}_i=v_i-q_i$$.
\end{defn}

\begin{defn} We call the $[\bar n^{\star},\bar v^{\star}]$ plane the (geometric) \emph{symbolic plane.} \end{defn}

\begin{defn}\label{def:rho}
Let
$$
\rho(M)= 1-\frac{\sum_{i=1}^{V} r_i}{2V}
$$
and we call $M$ $\rho$-regular.
In particular,  we call $M$ regular, semi-regular or irregular  if $\rho(M)=1, 1/2, 0$, respectively.
\end{defn}

\subsection{Trivial lower bounds on the cell and nodal degrees for large cell decompositions}
Here we regard the $E \to \infty$ limit of  a balanced normal tiling .
If $M$ is regular ($\rho =1$) then we have
\begin{equation}\label{regular}
\bar v = \bar v^{\star}, \quad
\bar n = \bar n^{\star}.
\end{equation}
Equation (\ref{meuler5}) implies
\begin{equation}
\bar n, \bar v \geq 2,
\end{equation} which, via (\ref{regular}) is equivalent to
\begin{equation}
\bar n^{\star}, \bar v^{\star} \geq 2.
\end{equation}
If, in addition to being regular, $M$  is also {  { combinatorially equivalent to a}} convex mosaic then we have
\begin{equation}
\bar n^{\star}, \bar v^{\star} \geq 3.
\end{equation}
All the listed bounds are trivial. Also, because of the inverse functional relationship (\ref{meuler5})
between $\bar n$ and $\bar v$,  the \emph{lower} bound of one of the combinatorial averages
is tied to the \emph{upper} bound of the other combinatorial average.

If $M$ is not regular $(\rho < 1)$ then there is no general functional relationship between $\bar n$ and $\bar v$.
Based on Definitions \ref{def:geom_node} and \ref{def:geom_cell} for non-regular tilings we have
\begin{equation}\label{non-regular}
\bar v  > \bar v^{\star}, \quad
\bar n >  \bar n^{\star}
\end{equation}
and we can also see that the corner degrees attain their respective minima if
$M$ is irregular, i.e. we have $\rho =0$.
Our next goal is to find the lower bound for $\bar v^{\star}$ for irregular normal tilings.

\subsection{Proof of Theorem \ref{main}}
\begin{proof}
Let $M$ be a regular normal tiling  with $F$ faces (cells), $V$ vertices (nodes) and $E$ edges and respective cell and nodal combinatorial degrees given in equation (\ref{ave}).
Let us \emph{deregularize} $M$ by transforming generic edge-pairs into degenerate ones. We perform this by local, continuous transformation in the vicinity of the nodes. Assume that by such local transformations we create $2V(1-\rho)$ degenerate pairs to obtain
the $\rho$-regular tiling $M(\rho)$. While the combinatorial degrees remain constant,
the corner degrees can be computed as a function of $\rho$:
\begin{equation}\label{iaverage}
\bar v^{\star}(\rho)= \frac{2E-2(1-\rho)V}{F}=\bar v \left(\frac{\bar n-2(1-\rho) }{\bar n}\right),
\end{equation}
$$
 \bar n^{\star}(\rho)=\frac{2E-2(1-\rho)V}{V}=\bar n-2(1-\rho).
$$
We can immediately see that
\begin{equation}\label{vn0}
\frac{\bar v}{\bar n}=\frac{\bar v^{\star}(\rho)}{\bar n^{\star}(\rho)},
\end{equation}
so the evolution path of topological deregularization on the $[\bar n, \bar v]$ symbolic plane
is a straight line passing through the origin. By subsitution $\rho=0$ into (\ref{iaverage}) we obtain
\begin{equation}\label{iaverage0}
 \bar n^{\star} \geq \frac{2E-2V}{V}=\bar n-2,
\end{equation}
\begin{equation}\label{iaverage1}
\bar v^{\star} \geq \frac{2E-2V}{F}=\bar v \left(\frac{\bar n-2 }{\bar n}\right).
 \end{equation}
Equation (\ref{meuler4}) is equivalent to
\begin{equation}\label{p1}
\bar v = \frac{2E\bar n}{\bar n(E+\chi)-2E}.
\end{equation}
Substituting (\ref{p1}) into (\ref{iaverage1}) yields the formula of Theorem \ref{main}.
\end{proof}

\begin{figure}[ht]
\begin{center}
\includegraphics[width=\textwidth]{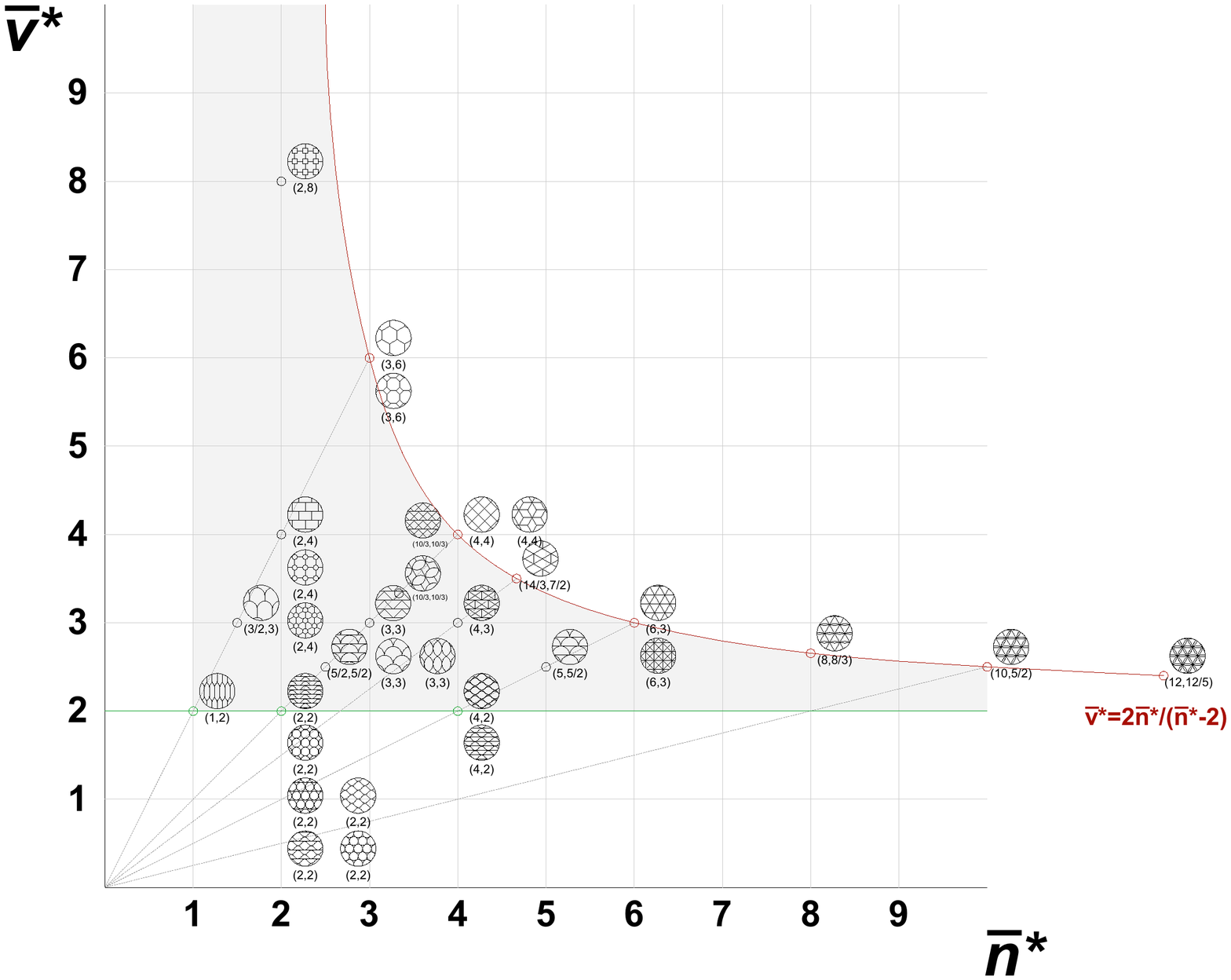}
\caption{Illustration of infinite patterns on the symbolic plane $[\bar n ^{\star}, \bar v^{\star}]$.
Observe that combinatorially equivalent tilings are connected by rays passing through the origin, illustrating formula (\ref{vn0}).
 }\label{fig:CW1}
\end{center}
\end{figure}

\section{Related issues and discussion}\label{3D}

\subsection{The monohedral case in $d=2$ dimensions}\label{ss:monohedral}
Here we discuss the special case where the tiling {{of the Euclidean plane}} consists of congruent  cells $\mathcal{C}$. Our next argument requires that at every smooth point of the boundary of $\mathcal{C}$ the signed curvature $\kappa$ exists. To guarantee this property we assume here that the smooth part of the boundary is finite union of two-times continuously differentiable arcs and for brevity, we call such tilings $C^2$-tilings. Theorem \ref{main} implies
\begin{cor}\label{cor1}
In case of monohedral, normal $C^2$-tilings {{of the Euclidean plane}}, for the corner degree of the cell we have $v ^{\star}\geq 2$.
\end{cor}
Below we give a proof for Corollary \ref{cor1} which is independent of the proof of Theorem \ref{main}.
\begin{proof}
Let us parametrize the circumference of the cell $\mathcal{C}$ by the arclength $s$ on the unit interval $I\equiv [0,1]$, so we have  $s \in [0,1]$. The boundary of the cell $\mathcal{C}$ is piecewise $C^2$-smooth. We will consider $I$ as the union of the (open) set of smooth points $I_S$ and the (discrete) set of combinatorial
nodes $N_i$, $i=1,2,\dots v$:
\begin{equation}\label{union}
I=I_S \bigcup \{N_i\}.
\end{equation}
 We denote the angle of the tangent by $\alpha(s)$ and at the nodes we will denote the finite non-negative angles
by $\Delta \alpha _i, (i=1,2,\dots v)$. We consider every combinatorial node to be a non-smooth point of the boundary of $C$.
The integral $K$ of the (signed) scalar curvature $\kappa$ can be written as
\begin{equation}\label{int01}
    K=\int_{I_S}\kappa(s) ds + \sum_{i=1}^{v}\Delta \alpha _i=2\pi.
\end{equation}
We observe that the discrete angle increments $\Delta \alpha_i$ differ from zero \emph{only} at corners. To simplify our computation, we re-assign indices $i \in [1,2, \dots v^{\star}]$ to these non-zero relative angles and so, instead of (\ref{int01}) we may write:
\begin{equation}\label{int1}
    K=\int_{I_S}\kappa(s) ds + \sum_{i=1}^{v^{\star}}\Delta \alpha _i=2\pi.
\end{equation} Since the tiling is monohedral, the neighbor cells are congruent to the investigated cell $C$.
Now we observe that for every smooth part $\gamma $ of the boundary there is at least one orientation preserving congruence $f_{\gamma}: \mathcal{M}\to \mathcal{M}$ which sends the cell  $C_\gamma$, overlapping with $C$ along the arc $\gamma$, to $C$. Let $\gamma '$ be the image of $\gamma $ by $f_\gamma$. If $\gamma '=\gamma$ then the value of the integral (\ref{int1}) on $\gamma$ is zero.  If this is not the case then the value of the integral on the union $\gamma \cup \gamma '$ is zero. Since $f_\gamma (C)$ is the neighbour of $C$ along the arc $\gamma '$, if there exists another smooth part $\gamma ^\star$ for which $f_{\gamma^\star}(\gamma^\star)=\gamma'$ then we have $(f_\gamma^{-1}\circ f_{\gamma^\star})(C)=C$ and so $f_\gamma^{-1}\circ f_{\gamma^\star}$ is an orientation preserving symmetry transformation of $C$. Thus $\gamma '':=f_\gamma^{-1}\circ f_{\gamma^\star}(\gamma ')$ is congruent to $\gamma'$. If $\gamma''\ne\gamma'$  then on the union of $\gamma$, $\gamma^\star$, $\gamma'$ and $\gamma ''$ the value of the integral is zero. Let us regard the arc $\gamma '$ with endpoints $P$ and $Q$.  If $\gamma ''=\gamma '$ then we have two cases: either the endpoints $P$, $Q$ of the arc $\gamma'$ remain fixed or they are interchanged. In both cases, the orientation preserving congruence $f_\gamma^{-1}\circ f_{\gamma^\star}$ fixes the segment $PQ$ and either $f_\gamma^{-1}\circ f_{\gamma^\star}$ is the identity (which is a contradiction) or it is a reflection at the center of $PQ$. This latter possibility means that the value of the integral is zero in $\gamma'$ and also in $\gamma$. It is easy to see by an inductive argument that on the whole smooth part the integral is zero and so 

\begin{equation}\label{int2}
    K=\int_{I_S}\kappa(s) ds + \sum_{i=1}^{v^{\star}}\Delta \alpha _i={ {\sum_{i=1}^{v^{\star}}\Delta \alpha _i=}}2\pi.
\end{equation} As $\Delta \alpha _i \leq \pi$, this implies that for the number $v^{\star}$ of  corners in $C$  we have $v^{\star}\geq 2$ and thus it proves the Corollary.
\end{proof}
\begin{rem}\label{rem5}
It is easy to see that the proof of Corollary \ref{cor1} also implies that for the normal, periodic tilings of the Euclidean plane we have $\bar v^{\star}\geq 2.$
\end{rem}

\subsection{Construction of a monohedral tiling with no vertices in 3 dimensions}
In 3 dimensions the deregularization algorithm, presented in the proof of Theorem \ref{main}, would work in an analogous manner as in 2 dimensions, since in any dimension $d$, at any point, the number of
hypercells (of dimension $d-1$) tangent to each other is exactly 2. This algorithm
reduces the nodal degree by 2, and thus we expect  (\ref{iaverage1}) to remain valid.
However, unlike in 2 dimensions, in the 3D case this equation will not provide the actual lower bound on $\bar v^{\star}$:
we have to consider that in 3 dimensions one can use a different deregularization algorithm as well, where 
not vertices but edges are being smoothed. The general formulae for the second algorithm appear to be less trivial 
as smoothing of edges is not a local operation.
Nevertheless, it is not difficult to establish a sharp lower bound in 3 dimensions:
\begin{figure}[h!]
\begin{center}
\includegraphics[width=0.8\textwidth]{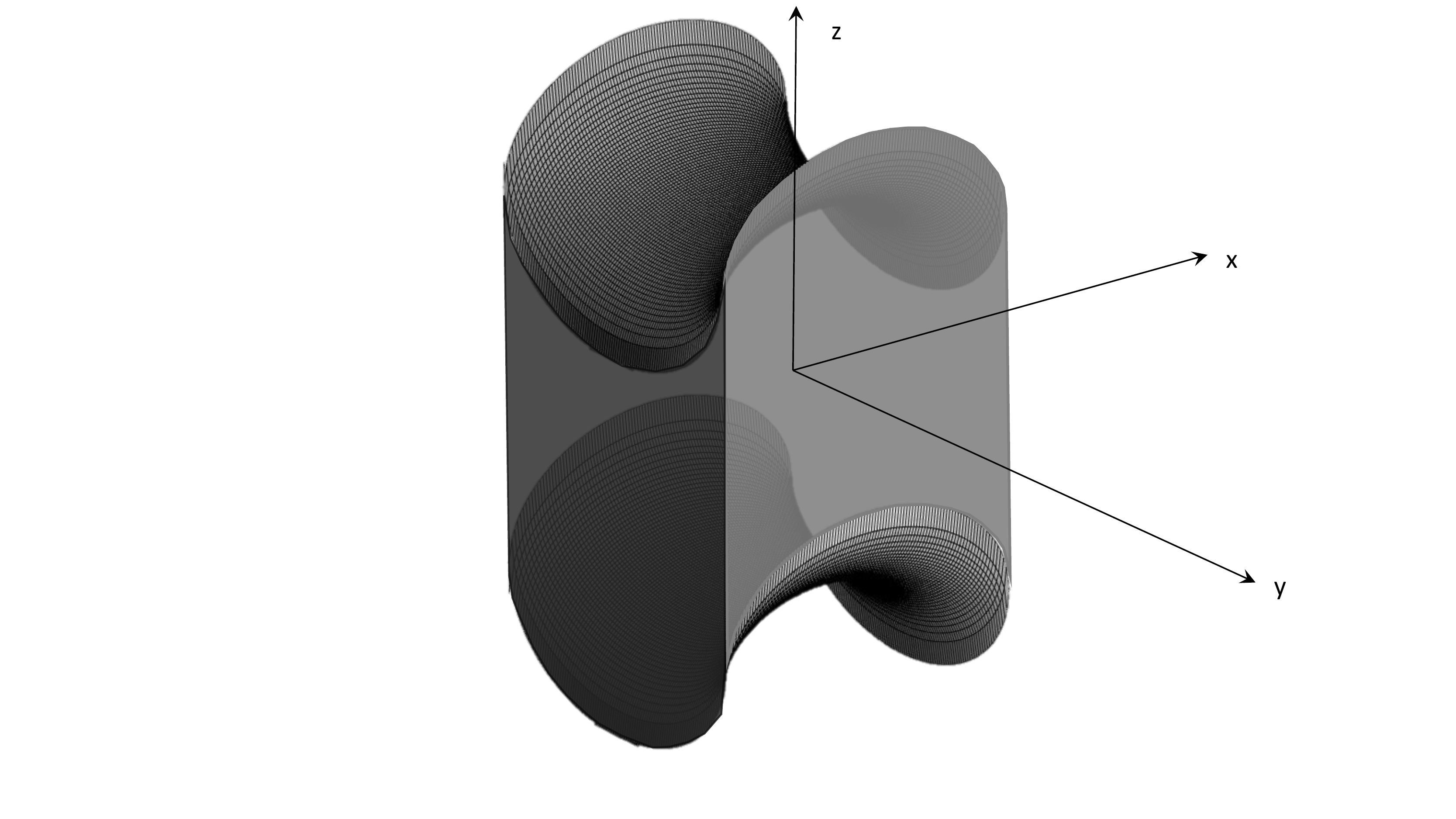}
\caption{{{Construction of a $v^{\star}=0$ monohedral tiling in $d=3$ dimensions. The rectangular prism has
its vertical edges at $(x,y)=\pm 1$ and the saddle-like surface at the top and the bottom are given as
$z=\sqrt{1-x^2}-\sqrt{1-y^2} \pm 1$.}}}\label{fig:00}
\end{center}
\end{figure}
\begin{lem}\label{3dmono}
In $d=3$ dimensions
$\bar v^{\star} \geq 0.$ 
\end{lem}
\begin{proof} We prove Lemma \ref{3dmono} by showing a monohedral tiling in 3D with $v^{\star}=0$, illustrated in Figure \ref{fig:00}.  \end{proof}

\bibliographystyle{plain}
\pagebreak
\bibliography{tiling}

\begin{thebibliography}{10}

\bibitem{adler_fracture_network_book}
P.~M. Adler and J-F. Thovert.
\newblock {\em Fractures and Fracture Networks}.
\newblock Springer, Dordrecht, 1999.

\bibitem{arch1}
W.~Chang.
\newblock Application of tessellation in architectural geometry design.
\newblock {\em E3S Web of Conferences, ICEMEE 2018}, 38 (03015), 2018, DOI
  https://doi.org/10.1051/e3sconf/20183803015.

\bibitem{Plato}
G{\'a}bor Domokos, Douglas~J. Jerolmack, Ferenc Kun, and J{\'a}nos
  T{\"o}r{\"o}k.
\newblock Plato{\textquoteright}s cube and the natural geometry of
  fragmentation.
\newblock {\em Proceedings of the National Academy of Sciences},
  117(31):18178--18185, 2020.

\bibitem{goering_columnar_2008}
L.~Goehring and S.~W. Morris.
\newblock Scaling of columnar joints in basalt.
\newblock {\em Journal of Geophysical Research: Solid Earth}, 113(B10), 2008.

\bibitem{goering_evolv_pattern_2013}
Lucas Goehring.
\newblock Evolving fracture patterns: columnar joints, mud cracks and polygonal
  terrain.
\newblock {\em Philosophical Transactions of the Royal Society A: Mathematical,
  Physical and Engineering Sciences}, 371(2004):20120353, 2013.

\bibitem{map1}
C.~Gold.
\newblock Tessellations in gis: Part ii–making changes.
\newblock {\em Geo-spatial Information Science}, 19:157 -- 167, 2016.

\bibitem{Grunbaum-Shepard}
B.~Gr\"unbaum and G.C. Shepard.
\newblock {\em Tilings and Patterns}.
\newblock Freeman and Co., New York, 1987.

\bibitem{cell1}
E.~Heller and E.~Fuchs.
\newblock Tissue patterning and cellular mechanics.
\newblock {\em Journal of Cell Biology}, 211:219 -- 231, 2015.

\bibitem{Horne}
C.E. Horne.
\newblock {\em Geometric Symmetry in Patterns and Tilings}.
\newblock Woodhead Publishing, 2000.

\bibitem{urban1}
B.~Jiang.
\newblock A topological pattern of urban street networks: Universality and
  peculiarity.
\newblock {\em Physica A}, 384:647 -- 655, 2007.

\bibitem{Kazanci}
D.~Kazanci and A.~Vince.
\newblock A property of normal tilings.
\newblock {\em Am. Math. Monthly}, 111:813 -- 816, 2004.

\bibitem{cell2}
E.~Laruelle, N.~Spassky, and A.~Genovesio.
\newblock Unraveling spatial cellular pattern by computational tissue
  shuffling.
\newblock {\em Communication Biology}, 3, 2020.

\bibitem{Senechal}
D.~Schattschneider and M.~Senechal.
\newblock {\em Tilings. Chapter 3 in: Discrete and Computational Geometry}.
\newblock CRC Press, 2004.

\bibitem{steacy_automaton_1991}
S.~Steacy and C.~Sammis.
\newblock An automaton for fractal patterns of fragmentation.
\newblock {\em Nature}, 353:250--252, 1991.

\end{thebibliography}

\end{document}